\documentclass{amsart}
\usepackage[margin=3.5cm]{geometry}
\usepackage{amsmath}
\usepackage{amsthm,amsmath,verbatim,amssymb}
\usepackage{color}
\usepackage{ifpdf}
\ifpdf
 \def\outdrv{pdftex}
\else
 \def\outdrv{dvipdfm}
\fi
\usepackage[\outdrv]{pict2e}
\ifpdf
\usepackage{epstopdf}
\fi
\newtheorem{lem}{Lemma}

\newtheorem{prop}{Proposition}
\newtheorem{thm}{Theorem}
\newtheorem*{thm*}{Theorem}

\newcommand{\CP}{\mathbb{CP}}

\newcommand{\kahler}{K\"ahler }
\renewcommand{\d}{\partial}

\newcommand{\R}{\mathbb{R}}

\newcommand{\C}{\mathbb{C}}
\def \to {\rightarrow}
\newcommand{\D}{\mathbb{D}}
\newcommand{\E}{\mathbb{E}}

\begin{document}\title []
{Critical values of Gaussian $SU(2)$ random polynomials}
\author[Renjie Feng]{Renjie Feng}
\author[Zhenan Wang]{Zhenan Wang}

\address{Department of Mathematics and Statistics, Mcgill University, Montreal, QC, Canada}
\email{renjie@math.mcgill.ca}
\address{Department of Mathematics, Northwestern  University, Evanston, IL 60208, USA}
\email{zn\_wang@math.northwestern.edu}

\date{\today}
\maketitle
\begin{abstract}In this article, we will get the estimate of the expected distribution of critical values  of Gaussian $SU(2)$ random polynomials as the degree is large enough. The result about the expected density is a direct application of the Kac-Rice formula. The critical values
will accumulate at infinity, then we will study the rate of this convergence and its rescaling limit as $n\to\infty$.
\end{abstract}

\section{Introduction}
Random polynomials and random holomorphic functions are studied as
ways to gain insight for problems arising in string theory and
analytic number theory \cite{DSZ, H, S}. In \cite{K}, Kac studied
and determined a formula for the expected distribution of zeros of
some real Gaussian random polynomials. His work was generalized to
complex random polynomials and random analytic functions throughout
the years, we refer to \cite{BSZ, DPSZ, EK, HKPV, ST} for more
backgrounds and results.

\subsection{$SU(2)$ polynomials}

When the random polynomial is defined invariant with respect to some
group action, the problem can turn out to be particularly
interesting, we refer \S 2.3 in \cite{HKPV} for examples. In this
article, we will study a special family: the Gaussian $SU(2)$
 random polynomials.   This is
of particular interest in the physics literature as the zeros describe a random
spin state for the ~Majorana~ representation (modulo phase) on the unit sphere \cite{H}.

Given a probability space $\Omega$ and $\{a_j\}_{j=0}^\infty$ a collection of i.i.d complex random variables with density $\frac{1}{\pi}e^{-|z|^2}$ on it,  the family of $SU(2)$ random polynomials is defined as
\begin{equation}\label{su2}p_n(z)=\sum_{j=0}^n a_j \sqrt{ {{n}\choose{j}}}z^j.\end{equation}
Although this polynomial is defined on $\C$, we may also view it as an
analytic  function on  $\CP^1 = \C\cup{\infty}$ with a pole at
$\infty$.

 Various
properties of the zeros of random $SU(2)$ polynomials have been studied such as the distribution of zeros and
the two points correlation function \cite{BSZ, H}. First, zeros of this polynomials are uniformly distributed on
$S^2\cong \CP^1$  with respect to the Fubini-Study metric, i.e., the average
distribution of zeros is invariant under the $SU(2)$ action on $\CP^
1$ \cite{HKPV}. To be more precise,
let's denote $$\mathcal Z_{p_n}=\sum_{z\in\CP^1:\,\,p_n(z)=0}\delta_z$$ as the empirical measure of zeros of Gaussian $SU(2)$ random polynomials and define the pairing
$$\langle \mathcal Z_{p_n},\phi\rangle =\sum_{z\in\CP^1:\,\,p_n(z)=0} \phi(z) \,\,\,\, \mbox{where}\,\,\, \phi\in C^\infty(\CP^1).$$
We define the expectation
$$\langle \E \mathcal Z_{p_n}, \phi\rangle:=\E \langle \mathcal Z_{p_n},\phi\rangle=\frac1{\pi^{n+1}}\int_{\C^{n+1}}\left(\sum_{z\in\CP^1:\,\,p_n(z)=0} \phi(z)\right)e^{-|a|^2}d\ell_{a_0}\cdots d\ell_{a_n},$$
where $d\ell_{a_j}=\frac 1{2i}da_j\wedge d\bar a_j$ is the Legesgue measure on $\C$.

Then the expected density of zeros is calculated in \cite{BSZ} as
$$\E \mathcal Z_{p_n}=n\omega_{FS},$$ in the sense that,
$$\E\langle \mathcal Z_{p_n},\phi\rangle =n\int_{\CP^1}\phi\omega_{FS} \,\,\,\, \mbox{where}\,\,\, \phi\in C^\infty(\CP^1),$$
where $\omega_{FS}$ is the Fubini-Study form on $\CP^1$ \cite{GH}.

We can also study the two points correlation function of zeros of $SU(2)$ polynomials and its scaling property. We define the two points correlation function as \cite{BSZ}
 $$ K_n(z,w):=\E \left(Z_{p_n}(z)\otimes Z_{p_n}(w)\right),$$
such that for any smooth test function $\phi_1(z)\otimes\phi_2(w) $, we have the pairing
$$\left\langle K_n(z,w),\phi_1(z)\otimes\phi_2(w)\right\rangle=\E \left(\langle \mathcal Z_{p_n}, \phi_1\rangle\right)\left(\langle \mathcal Z_{p_n}, \phi_2\rangle\right).$$
If we scale the two points correlation function by a factor $\frac 1 {\sqrt n}$, then we have
 $$K_n(\frac{z}{\sqrt n},\frac{w}{\sqrt n})=\frac{(\sinh t^2+t^2)\cosh t-2t\sinh t}{\sinh t^3}+O(\frac 1 {\sqrt n}),$$
 where $t=\frac {|z-w|^2}{2}$ and $|z-w|$ is the geodesic distance of $z$ and $w$ on $\CP^1$.
 It's easy to see
 $$K_n(\frac{z}{\sqrt n},\frac{w}{\sqrt n})=t-\frac 2 9 t^3+O(t^5) \,\,\,\mbox{as}\,\,\, t\to 0,$$
 which implies zeros repel each other. We refer to \cite{BSZ, H} for more details.

\subsection{Main results}
In this article, we will study the expected distribution of nonvanishing critical values of $|p_n|$ as $n$ tends to infinity.

Note that the modulus $|p_n|$ is a subharmonic function, thus there is no local maximum; local minimum are all zeros and thus nonvanishing critical values are obtained only at saddle points \cite{FS}. Hence the expected density of nonvanishing critical values of $|p_n|$ we study in this article is in fact the expected density of values of saddle points of $|p_n|$.

The nonvanishing critical values of $|p_n|$ are obtained at points
   \begin{equation}\label{set}\{z\in\C:\,\,\,p_n'=0 \,\,\mbox{and}\,\,\, p_n\neq 0\}.   \end{equation}
 For a random polynomial $p_n$, it has no repeated zeros almost surely, 
 which implies that
 the set (\ref{set}) is almost surely equivalent to
   \begin{equation}\{z\in\C:\,\,\,p_n'=0\},\end{equation}  i.e., (nonvanishing) $|p_n|$ and $p_n$ have the same critical points almost surely.

   Hence, we will first get the expected density of critical values of $p_n$ in Theorem \ref{main}, as a direct consequence, we can apply the polar coordinate to get the expected density of nonvanishing critical values of $|p_n|$ in Theorem \ref{main2}.

We denote the empirical measure of critical values of $p_n$ as
\begin{equation}\label{empirical}\mathcal C_{p_n}=\sum_{z: \,\, p'_n(z)=0}\delta_{p_n(z)}.\end{equation}
 We now define the pairing
\begin{equation}\label{pair}\langle \mathcal C_{p_n}, \phi\rangle =\sum_{z: \,\, p'_n(z)=0}\phi(p_n(z)), \,\,\,\,\,\, \forall \phi(x)\in C_c^\infty(\R^2),\end{equation}
where  $C_c^\infty(\R^2)$ is the space of smooth functions on $\R^2$ with compact support.

We denote $\mathbb D_{p_n}(x)$ as the expected density of critical values of $p_n$
 in the sense that
\begin{equation}\label{expect}\E\langle \mathcal C_{p_n}, \phi\rangle =\int_\C \phi(x)\mathbb D_{p_n}(x)d\ell_x , \,\,\,\,\,\,\forall \phi(x)\in C_c^\infty(\R^2),\end{equation}
whereas $d\ell_ x$ is the ~Lebesgue~ measure of $\C$.

Those definitions also apply to the empirical measure of the nonvanishing
critical values of $|p_n|$ which is
\begin{equation}\label{em3}\mathcal C_{|p_n|}=\sum_{z: \,\,
p'_n=0}\delta_{|p_n|},\end{equation} which is a measure defined on the
nonnegative real line $\R_+$.

We define its expectation as
\begin{equation}\langle \E\mathcal C_{|p_n|}, \phi\rangle:=\E\langle \mathcal C_{|p_n|}, \phi\rangle=
\int_0^\infty\phi(x)\mathbb D_{|p_n|}dx,\,\,\,\forall\,\,\phi(x)\in
C^\infty_c(\R_+),\end{equation} where $dx$ is the ~Lebesgue~ measure
on $\R$.

In this article, we will first get the exact formula for the expected density $\mathbb D_{p_n}$ in the Proposition \ref{Dn} by the Kac-Rice formula (see section \S \ref{KR}), then we study the asymptotic behavior of $\D_{p_n}$ as $n\to \infty$. Our main results are
\begin{thm} \label{main}The expected density $\mathbb D_{p_n}$ of the empirical measure $\mathcal C_{p_n}$ of the critical values of $p_n$ satisfies the estimate,
\begin{equation} \label{mainequation} \mathbb D_{p_n}=\frac {1-e^{-|x|^2}}{\pi|x|^2}+\frac 1\pi\int _0^1 e^{-(s-s\log s)|x|^2}ds+o(1) \,\,\,\mbox{as}\,\,\, n\to \infty,
\end{equation}
 for  any $x\in \C$.
\end{thm}


As proved in Proposition \ref{Dn}, the density $\D_{p_n}d\ell_x$ only depends on $|x|$, i.e., the modulus of $|p_n|$, thus we can rewrite it as $\D_{p_n}(|x|)|x|d|x|d\theta$ under the polar coordinate. If we integrate on $\theta$ variable, then  \begin{equation}\label{relation}\D_{|p_n|}=\int_0^{2\pi} \D_{p_n}(|x|)|x|d\theta=2\pi|x| \D_{p_n}(|x|)\end{equation} will be the density of critical values of $|p_n|$. Thus we get
\begin{thm} \label{main2}The expected density $\mathbb D_{|p_n|}$ of the empirical measure $\mathcal C_{|p_n|}$ of the nonvanishing critical values of $|p_n|$ satisfies the estimate,
\begin{equation} \label{mainequation2} \mathbb D_{|p_n|}(x)=\frac {2(1-e^{-x^2})}{x}+2x\int _0^1 e^{-(s-s\log s)x^2}ds+o(1) \,\,\,\mbox{as}\,\,\, n\to \infty,
\end{equation}
for any $x\in \R_+$.
\end{thm}
The decay of $\mathbb D_{|p_n|}(x)$ is of order $1/x$ as $x$ goes to infinity, thus the total mass on the interval $[a,\infty)$ is infinity for any $a>0$, i.e., the critical values will accumulate at infinity as $n\to\infty$. In order to study the rate of this accumulation,  we  consider  the distribution function $F_n(x)$ of the following probability density $$\frac{\D_{|p_n|}}{n-1}=\frac1{n-1}\E (\sum_{p_n'=0}\delta_{|p_n|}).$$

Next, we will show that the critical values are spreading out    exponentially.
\begin{thm}\label{grow}
For any fixed $ \epsilon>0$, $F_n(e^{n^\frac{1-\epsilon}{2}})\to 0$ and $F_n(e^{n^\frac{1+\epsilon}{2}})\to 1$ as $n\to\infty$.
\end{thm}

 Then the modulus of critical values of $p_n$ will mainly concentrate in the interval $[e^{n^\frac{1-\epsilon}{2}},e^{n^\frac{1+\epsilon}{2}}]$ as $n$ large enough. Thus we need to consider the following rescaled probability density to get more information about this convergence
$$R_n(x)=(F_n(e^{\frac{nx}2}))'.$$ Then we prove that $R_n(x)$ satisfies the following rescaled limit
\begin{thm}\label{rescaling}
\begin{equation*}
\lim_{n\to\infty}R_n(x)=
\begin{cases} e^{-x} & \text{if $x>0$,}
\\
 \lim_{n\to\infty}\mathbb D_{p_n}(1)&\text{if $x=0$,}
 \\
 0&\text{if $x<0$,}
\end{cases}
\end{equation*}
where $\lim_{n\to\infty}\mathbb D_{p_n}(1)$ is the constant given
by the leading term in \eqref{mainequation}  evaluated at $1$.
\end{thm}

\subsection{Further remarks}
First note that our setting is different from the one in \cite{DSZ}. For example, in \cite{DSZ}, critical points of $SU(2)$ polynomials are defined to be the points $$\{z\in\CP^1:\,\, \nabla'p_n=0\},$$
where $\nabla'=\frac \d{\d z}-\frac{n\bar zdz}{1+|z|^2}$ is the smooth Chern connection on the line bundle $\mathcal O(n)\to \CP^1$ with respect to the ~Fubini~-Study metric and $p_n$ is a global holomorphic section of the line bundle $\mathcal O(n)\to \CP^1$ \cite{GH}. By choosing such smooth Chern connection, the expected distribution of critical points is also invariant under the $SU(2)$ action \cite{DSZ}.  But in this article, the critical points are defined by the usual derivative $$\{z\in\C:\,\, \frac{\d p_n}{\d z}=0\}.$$ In fact, the derivative $\frac \d{\d z}$ is a meromorphic flat Chern connection on $\mathcal O(n)\to \CP^1$ with a pole at $\infty$. Under this setting, the expected density of critical points is not $SU(2)$ invariant, we refer to \cite{Ha} for more details.

 Our second remark is as following. In \cite{FZ}, the authors studied the expected density of nonvanishing critical values of the pointwise norm of Gaussian random holomorphic sections of the positive holomorphic line bundle over compact \kahler manifolds. Now let's briefly explain the main result in \cite{FZ} and compare it with Theorem \ref{main2}.  Take Gaussian $SU(2)$ random polynomials (sections) $p_n$ for example.  We equip the line bundle $\mathcal O(n)\to \CP^1$ with a Hermitian metric $h^n=e^{-n\phi}$ where $\phi=\log ( 1+|z|^2)$ is the \kahler potential of Fubini-Study metric. Then the pointwise $h$-norm of the holomorphic section $|p_n|_{h^n}=|p_n|e^{-\frac {n\phi} 2}$ is global defined on $\CP^1$ \cite{GH} and hence the critical points of $|p_n|_{h^n}$ is defined as
$$\Sigma_n=\{z\in \CP^1: \,\, \frac {\d |p_n|_{h^n}}{\d z}=0\}.$$
We define the (normalized) empirical measure of critical values of $|p_n|_{h^n}$ as
$$\mathcal C_{|p_n|_{h^n}}:=\frac 1 n\left(\sum_{z\in \Sigma_n}\delta_{|p_n|_{h^n}}\right),$$
which is also a measure defined on $\R_+$.

Then the expectation of $\mathcal C_{|p_n|_{h^n}}$ satisfies the estimate
\begin{equation}\label{timed}\E \mathcal C_{|p_n|_{h^n}}=x\left(2x^2-4+8e^{-\frac{x^2}2}\right) e^{-x^2}+O(\frac 1 n),\,\,\,\,x\in \R_+,\end{equation}
as $n$ large enough. In fact, this estimate is universal: it holds on any Riemannian surfaces \cite{FZ}.

Thus the (normalized) density $\E\mathcal C_{|p_n|_{h^n}}$ is
decaying exponentially  as $x$ is large enough which is quite different from
the behavior of (non-normalized) density $\E\mathcal C_{|p_n|}$ in
Theorem \ref{main2}. This is mainly because of the connection we
choose: the usual derivative $\frac \d{dz}$ in this article is a
meromorphic flat connection on $\CP^1$ with a pole at $\infty$ while
in \cite{FZ}, the proof of (\ref{timed}) relies on a choice of
smooth Chern connection $\nabla'=\frac \d{\d z}-\frac{n\bar zd
z}{1+|z|^2}$.

\bigskip
\textbf{Acknowledgements}:
We would like to thank S. Zelditch for suggesting this problem and many helpful discussions. We also want to thank the referee for many valuable
comments in the original version.


\section{Kac-Rice formula}\label{KR}
In this section,  we first review the Kac-Rice formula for a stochastic process, referring to \cite{AT, K, R} for more details. Then we generalize the formula to the expected distribution of critical values of $p_n$.

The Kac-Rice formula is as follows: let $f(z)$ be a real valued stochastic process indexed by a compact interval $I\subset \R$.  Then the Kac-Rice formula for
the expected number of zeros is$$\label{kr}\E\#\{z\in I: \,\,f(z)=0\}=\int_I \int_\R |y|p_z(0,y)dydz,$$
where $p_z(0,y)$ is the joint density $p_z(x,y)$ of $(f, f')$ evaluated at $(0,y)$.
If $f$ is a Gaussian process, then the joint  density $p_z(x,y)$ is determined by the covariance matrix of $(f,f')$ \cite{AT}.

The proof of this formula is explained in more details in \cite{AT}.
The idea of the proof is based on the following observation
$$\#\{z\in I: \,\,f(z)=0\}=\int_I \delta_0(f(z))|f'(z)|dz.$$
We take expectation on both sides to get
\begin{align*}\E\#\{z\in I: \,\,f(z)=0\}=&\int_I  \int_{\R_y}\int_{\R_x} \delta_0(x)p_z(x,y)|y|dxdydz\\=&\int_I\int_\R |y|p_z(0,y)dydz.\end{align*}
Thus the expected density of zeros of $f$ is given by
\begin{equation}\label{realcase}\E \left(\sum_{z\in I:\,\, f(z)=0}\delta_z\right)=\left(\int_\R |y|p_z(0,y)dy\right)dz.\end{equation}
If $f(z)$ is a complex stochastic process indexed by a compact complex domain, the above formula reads
\begin{equation}\label{complexcase}\E \left(\sum_{z\in I:\,\, f(z)=0}\delta_z\right)=\left(\int_\C |y|^2p_z(0,y)d\ell_y\right)d\ell_z,\end{equation}
 where $d\ell_y$ and $d\ell_ z$ are Lebesgue measures on $\C$. Compared with (\ref{realcase}), we get $|y|^2$ since a $1$-dimensional complex random process is a $2$-dimensional real random process. In fact, this formula is  based on the definition of the delta function and the identity
$$\#\{z\in I: \,\,f(z)=0\}=\int_I \delta_0(f(z))\frac 1{2i }df\wedge d\bar f=\int_I \delta_0(f(z))|f'|^2d\ell_ z.$$
The formula arises when we take expectation on both
sides.

\subsection{Kac-Rice formula: Revisit}
In this subsection, let's get the formula for the expected density of critical values of a (real or complex) stochastic process $f$ by the method of Kac-Rice.

For simplicity, let's first consider a smooth real gaussian process $f\in C^\infty(I)$ where $I$ is a compact subset in $\R$.

 Let $\Theta\subset \R$  be a compact subset. Let's denote the set of critical values in $\Theta$ as
$$\mathcal C_\Theta=\{z\in I: \,\,f(z)\in \Theta, \,\,\,f'(z)=0\}.$$
Let's denote the measure  $\mu(x)dx$ on $\Theta$ as
$$\mu(x)dx=\E\left(\sum_{z\in \mathcal C_\Theta}\delta_{f(z)}\right),$$
in the sense that,
$$\E\left(\langle\sum_{z\in \mathcal C_\Theta}\delta_{f(z)},\phi\rangle\right)=\int_\Theta \phi \mu(x)dx,$$
where $\phi$ is any smooth test function defined on $\Theta$.

Then we have the following lemma
\begin{lem}\label{easycase}Let's denote $p_z(x,y,\xi)$ as the joint density of $(f,f',f'')$ at $z$. Then $$\mu(x)dx=\left(\int_{I}\int_\R |\xi|p_z(x,0,\xi)d\xi dz\right)dx,$$
 where $dx$, $d\xi$ and $dz$ are ~Lebesgue~ measures on $\R$.
\end{lem}\begin{proof}We will first do a formal calculation.
\begin{equation}\label{formalcal} \langle \sum_{f\in\Theta,\,\, f'=0}\delta_{f(z)}, \phi(x)\rangle=\sum_{f\in\Theta,\,\, f'=0}\phi(f(z))=\int_I \chi_{\{f\in\Theta\}} \phi(f(z))\delta(f') df'.\end{equation}
By taking expectation on both sides and considering $df'=f''dz$, we have
\begin{align}\label{fc1} \E\left\langle \sum_{f\in\Theta,\,\, f'=0}\delta_{f(z)}, \phi(x)\right\rangle =&\int_{\R_x}\int_I\int_{\R_\xi}\int_{
\R_y} p_z(x,y,\xi)\chi_{\{x\in\Theta\}} \phi(x)\delta(y) |\xi|dy d\xi dzdx\\
\label{fc2}=&\int_{\Theta}\left(\int_I\int_{\R_\xi} p_z(x, 0,\xi) |\xi| d\xi dz \right)\phi(x)dx\\\label{fc3}=&\int_\Theta \phi(x)\mu(x)dx.\end{align}

This calculation requires justification in \eqref{formalcal}, \eqref{fc1} and \eqref{fc2}. The way to rigorously do that is to approximate the $\delta$ function by a sequence of simple functions and do a verbatim repetition of the proof in \cite[Theorem 11.2.3,Corollary 11.2.4]{AT}.

From \eqref{fc3} to the conclusion, one need to prove the density on both hand sides are continuous. 
 To prove the continuity of $\mu$, it is again repeating the argument in  \cite[Theorem 11.2.3, Corollary 11.2.4]{AT}, see \cite[Section 11.4]{AT} and \cite{AW} for details.
\end{proof}
In the proof of Lemma \ref{easycase}, we have assumed $I$ and $\Theta$ being compact subsets in $\R$. But the proof of Lemma \ref{easycase} can be generalized to the $SU(2)$ random polynomials $p_n$ which are a collection of complex Gaussian stochastic processes indexed by $\C$.

 The generalization of $\Theta$ to be $\C$ only requires picking up a sequence of discs centered at the origin with radius $m\in\{1,2,\ldots\}$ and taking limit in weak sense. And the generalization from $I$ to $\C$ is the same.

 However, we do need to modify the pairing by choosing the test functions $\phi(z)$ in the smooth compact supported space $\mathcal C_c^\infty(\R^2)$ in order to change the order of the integration on $\C$.
Following the proof of Lemma \ref{easycase}, we have
\begin{lem}\label{complexcase} The expected density of critical values of $p_n$ is,
\begin{equation}\label{density}\D_{p_n}d\ell_x=\left(\int_{\C}\int_\C |\xi|^2p_z(x,0,\xi)d\ell_\xi d\ell_z\right)d\ell_ x,\end{equation}
where $d\ell_x, d\ell_\xi$ and $d\ell_ z$ are Lebesgue measures on $\C$ and \begin{equation}\label{density}p_z(x,0,\xi)=\frac 1{\pi^3\det\Delta_z}\exp\left\{-\left\langle \begin{pmatrix}x\\ 0 \\ \xi \end{pmatrix}, \Delta_z^{-1}\begin{pmatrix}\bar x\\ 0 \\ \bar \xi \end{pmatrix}\right \rangle\right\}\end{equation} is the joint density of
$(p_n,p_n',p_n'')$ where $\Delta_z$ is the covariance matrix of $(p_n,p_n',p_n'')$. \end{lem}

\begin{proof}
The proof of this formula is the same as the one in Lemma
\ref{easycase}.

We start with a disk $U$ in place of $I$, take $\Theta\subset \C$ compact and write $\mathcal C_\Theta=\{z\in U: \,\,f(z)\in \Theta, \,\,\,f'(z)=0\}$ again. Then we have
\begin{equation}\label{formalcal_1} \langle \sum_{f\in\Theta,\,\, f'=0}\delta_{f(z)}, \phi(x)\rangle=\sum_{f\in\Theta,\,\, f'=0}\phi(f(z))=\int_U \chi_{\{f\in\Theta\}} \phi(f(z))\delta(f') (\frac{1}{2i}df'\wedge d\bar{f'}).\end{equation}

Taking expectation on both hand sides and noting $df'\wedge d\bar{f'}=|f''|^2dz\wedge d\bar{z}=|f''|d\ell_z$, we have
\begin{align*} \E\left\langle \sum_{f\in\Theta,\,\, f'=0}\delta_{f(z)}, \phi(x)\right\rangle =&\int_{\C_x}\int_U\int_{\C_\xi}\int_{
\C_y} p_z(x,y,\xi)\chi_{\{x\in\Theta\}} \phi(x)\delta(y) |\xi|^2d\ell_y d\ell_\xi d\ell_zd\ell_x\\
=&\int_{\Theta}\left(\int_U\int_{\C_\xi} p_z(x, 0,\xi) |\xi|^2 d\ell_\xi d\ell_z \right)\phi(x)d\ell_x\\=&\int_\Theta \phi(x)\mu(x)d\ell_x.\end{align*}

The justification process here is the same as in Lemma \ref{easycase}.
 
The lemma follows if we replace $f(z)$ by $p_n(z)$. Since $p_n$ is a Gaussian process,
  the joint density $p_z(x,y,\xi)$ is uniquely determined by the
covariance matrix of $(p_n,p_n',p_n'')$.
\end{proof}

\section{Proof of main Theorems}
\subsection{The Density $\D_{p_n}$}
In this subsection, we will derive the exact formula for $\D_{p_n}$  based on Lemma \ref{complexcase}. We prove,
\begin{prop}\label{Dn}The expected density of the empirical measure of $\mathcal C_{p_n}$ is given by the formula
\begin{equation}\label{dpn}\D_{p_n}=\frac{n-1}{\pi}\int_1^\infty \frac{n(r-1)+1}{r^{n+2}}e^{-\frac{n(r-1)+1}{r^n}|x|^2}dr.\end{equation}
Thus $\D_{p_n}$ is a function only depending on $|x|$.
\end{prop}

\begin{proof}
By Lemma \ref{complexcase}, in order to compute the expected density of critical values of $p_n$,  we first need to compute the covariance matrix of $(p_n,p_n',p_n'')$.

By definition, the covariance matrix of the Gaussian process $(p_n,p_n',p_n'')$ is given by \cite{AT, DSZ}
$$ \Delta=\begin{pmatrix} \E(p_n\overline{p_n}) & \E(p'_n\overline{p_n}) & \E(p_n''\overline{p_n}), \\
 \E(p_n\overline{p_n'})  & \E(p'_n\overline{p_n'}) & \E(p''_n\overline{p'_n}), \\
\E(p_n\overline{p''_n})  &\E(p'_n\overline{p''_n}) &\E(p''_n\overline{p''_n})
\end{pmatrix}.
$$
The covariance kernel for the Gaussian process $p_n$ is$$\E(p_n(z)\overline{p_n}(w)): =\Pi_n(z,w)=(1+z\bar w)^n.$$
Then we can express each entry in the covariance matrix as following
$$\E(p_n\overline{p_n})=\Pi_n(z,z)=(1+|z|^2)^n,$$
$$\E(p'_n\overline{p_n})=\frac {\partial\Pi_n(z,w)}{\partial z}|_{z=w}=n\bar{w}(1+z\bar{w})^{n-1}|_{z=w}=n\bar{z}(1+|z|^2)^{n-1},$$

$$ \E(p_n''\overline{p_n})=\frac {\partial^2\Pi_n(z,w)}{\partial^2 z}|_{z=w}=n(n-1)\bar{w}^2(1+z\bar{w})^{n-2}_{z=w}=n(n-1)\bar{z}^2(1+|z|^2)^{n-2},$$
$$\E(p'_n\overline{p'_n})=\frac {\partial^2\Pi_n(z,w)}{\partial z\partial {\bar w}}|_{z=w}=n(1 + z\bar{w})^{n-2}((n-1)z\bar{w} + 1+z\bar{w})|_{z=w}=n(n|z|^2+1)(1+|z|^2)^{n-2},$$

\begin{align*}\E(p_n''\overline{p_n'})=\frac {\partial^3\Pi_n(z,w)}{\partial^2 z\partial \bar w}|_{z=w}&=n(n-1)(1+z \bar{w})^{n-3}\bar{w}(nz\bar{w}+2)|_{z=w}\\&=n(n-1)(1+|z|^2)^{n-3}\bar{z}(n|z|^2+2),\end{align*}

\begin{align*}\E(p''_n\overline{p''_n})&=\frac {\partial^4\Pi_n(z,w)}{\partial^2 z\partial^2 \bar w}|_{z=w}\\&=
n(n-1)((n-2)(n-3)z^2\bar{w}^2(1+z\bar{w})^{n-4}+4(n-2)z\bar{w}(1+z\bar{w})^{n-3}+2(1+z\bar{w})^{n-2})|_{z=w}\\
&=n(n-1)(1+|z|^2)^{n-4}(n(n-1)|z|^4+4(n-1)|z|^2+2).\end{align*}
These show the covariance matrix is
$$ \Delta_z=(1+|z|^2)^{n}\begin{pmatrix} 1 & \frac{n\bar z}{1+|z|^2}& \frac{n(n-1)\bar z^2}{(1+|z|^2)^2}\\
 \frac{n z}{1+|z|^2} & \frac{n+n^2| z|^2}{(1+|z|^2)^2}& \frac{2n(n-1)\bar z+(n-1)n^2\bar z |z|^2}{(1+|z|^2)^3}\\
\frac{n(n-1) z^2}{(1+|z|^2)^2} &  \frac{2n(n-1)z+(n-1)n^2 z |z|^2}{(1+|z|^2)^3}& \frac{2n(n-1)+4n(n-1)^2|z|^2+n^2(n-1)^2|z|^4}{(1+|z|^2)^4}
\end{pmatrix}\;.
$$
Hence\begin{equation}\label{deter}\det \Delta_z= (1+|z|^2)^{3n}\frac{2n^3-2n^2}{(1+|z|^2)^6},\end{equation} which never degenerates when $n>1$.
We denote$$Q_z(x,\xi)=:\left\langle \begin{pmatrix}x\\ 0 \\ \xi \end{pmatrix}, \Delta_z^{-1}\begin{pmatrix}\bar x\\ 0 \\ \bar \xi \end{pmatrix}\right \rangle .$$
Then by direct computations, we rewrite (note we only need to calculate the four corner entries of the inverse matrix)   $$Q_z(x,\xi)=\frac{(1+|z|^2)^{2n}}{\det \Delta_z}\left\langle \begin{pmatrix}x \\ \xi \end{pmatrix}, \begin{pmatrix} \frac{n^5|z|^4+2n^4(|z|^2-|z|^4)+n^3(|z|^4-2|z|^2+2)-2n^2}{(1+|z|^2)^6}& \frac{(n^3-n^2)\bar z^2}{(1+|z|^2)^4}\\\frac{(n^3-n^2) z^2}{(1+|z|^2)^4}& \frac n {(1+|z|^2)^2}\end{pmatrix} \begin{pmatrix}\bar x\\  \bar \xi \end{pmatrix} \right\rangle.$$
We expand this expression and further rewrite $Q_z(x,\xi)$ as
\begin{equation}\label{vv}\frac {1}{2(1+|z|^2)^n}\left(\left|\sqrt{n^2-n}\bar z^2x+\frac 1{\sqrt{n^2-n}}\xi (1+|z|^2)^2\right|^2+2(n|z|^2+1)|x|^2\right).\end{equation}

By Lemma \ref{complexcase}, the expected density of critical values of $p_n$ is given by the formula
\begin{equation}\label{ddddd}\D_{p_n}(x)=\frac 1{\pi^3}\int_{\C}\int_\C \frac {e^{-Q_z(x,\xi)}}{\det \Delta_z}|\xi|^2d\ell_\xi d\ell_z.\end{equation}

Let's integrate $\xi$ variable first. Plug (\ref{vv}) into (\ref{ddddd}), we can rewrite (\ref{ddddd}) as
\begin{equation}\label{ttt}\D_{p_n}(x)=\frac 1{\pi^3}\int_{\C}K_z \frac {e^{-\frac{n|z|^2+1}{(1+|z|^2)^n}|x|^2}}{\det \Delta_z} d\ell_z,
\end{equation}
where $K_z$ is the following integral in $\xi$ variable
$$K_z=\int_\C \exp \left\{-\frac 1{2(1+|z|^2)^n}\left|\sqrt{n^2-n}\bar z^2x+\frac 1{\sqrt{n^2-n}}\xi (1+|z|^2)^2\right|^2\right\}|\xi|^2d\ell_\xi .$$
We will first make the exponent into a perfect square. We change variables $\xi\to \frac 1 {\sqrt{n^2-n}} \xi (1+|z|^2)^2$ to get
$$K_z=\frac {(n^2-n)^2}{(1+|z|^2)^8}\int_\C \exp \left\{-\frac 1{2(1+|z|^2)^n}\left|\sqrt{n^2-n}x\bar z^2+\xi\right|^2\right \}|\xi|^2d\ell_\xi .$$
Further changing variable $\xi\to \sqrt{n^2-n}x\bar z^2+\xi$ to get  $$K_z=\frac {(n^2-n)^2}{(1+|z|^2)^8}\int_\C \exp \left\{-\frac {|\xi|^2}{2(1+|z|^2)^n}\right\}\left|\xi-\sqrt{n^2-n}x\bar z^2\right|^2d\ell_\xi.$$
This turns into a Gaussian integral.
Noting that the first moment terms equal to zero after expanding the norm square, we have
$$K_z=\pi \frac {(n^2-n)^2}{(1+|z|^2)^8}\left[2(1+|z|^2)^n(n^2-n)|x|^2|z|^4+4(1+|z|^2)^{2n}\right].$$
If  we change variable $r=1+|z|^2$, we can rewrite
$$K_z=\pi \frac {(n^2-n)^2}{r^8}\left[2r^n(n^2-n)|x|^2(r-1)^2+4r^{2n}\right]$$
and
$$\det\Delta_z=r^{3n-6}(2n^3-2n^2), \,\,\,\, e^{-\frac{n|z|^2+1}{(1+|z|^2)^n}|x|^2}=e^{-\frac{n(r-1)+1}{r^n}|x|^2}.$$
Now we plug these two lines back into the formula of (\ref{ttt}) and use the polar coordinate $d\ell_z=\frac 1 2drd\theta$, integrate on $\theta$ variable, we can rewrite $\D_{p_n}$ as

\begin{equation}\label{ttttt}\D_{p_n}=\frac{n-1}{\pi}\int_1^\infty \frac{(n^2-n)r^n(r-1)^2|x|^2+2r^{2n}}{r^{3n+2}}e^{-\frac{n(r-1)+1}{r^n}|x|^2} dr.\end{equation}
There are two parts in the numerator, we integrate by part to simplify the first term in the numerator. Note that $$de^{-\frac{n(r-1)+1}{r^n}|x|^2}=e^{-\frac{n(r-1)+1}{r^n}|x|^2}[r^{-n-1}(n^2-n)(r-1)|x|^2]dr,$$
then the first part is equal to
\[ \begin{aligned}&\frac{n-1}{\pi}\int_1^\infty \frac{(n^2-n)r^n(r-1)^2|x|^2}{r^{3n+2}}e^{-\frac{n(r-1)+1}{r^n}|x|^2} dr\\=&\frac{n-1}{\pi}\int_1^\infty \frac{(r-1)}{r^{n+1}}de^{-\frac{n(r-1)+1}{r^n}|x|^2}\\=&\frac{n-1}{\pi}\int_1^\infty [\frac n{r^{n+1}}-\frac{n+1}{r^{n+2}}]e^{-\frac{n(r-1)+1}{r^n}|x|^2}dr.\end{aligned}\]
Hence the density (\ref{ttttt}) is further simplified to be
$$\D_{p_n}(x)=\frac{n-1}{\pi}\int_1^\infty \frac{n(r-1)+1}{r^{n+2}}e^{-\frac{n(r-1)+1}{r^n}|x|^2}dr,$$
which completes the proof.

\end{proof}






\subsection{Proof of Theorem \ref{main}}
Now we turn to the proof of our main Theorem \ref{main}.

We denote $t=\frac 1{r}$ and $$y_n(t)=\frac{n(r-1)+1}{r^n}=nt^{n-1}-(n-1)t^n,$$ then we have $t\in [0,1]$ and $y_n(t)\in [0,1]$ with $y_n(0)=0$ and $y_n(1)=1$.

Substituting $\frac{n(r-1)+1}{r^n}$ by $y_n(t)$, we rewrite $\D_{p_n}$ in Proposition (\ref{Dn}) as
 \begin{equation}\label{inte} \begin{aligned}\D_{p_n}=&\frac{n-1}{\pi}\int_1^\infty \frac{y_n(t)}{r^2}e^{-y_n(t)|x|^2}dr\\=&\frac{n-1}{\pi}\int_0^1
y_n(t)e^{-y_n(t)|x|^2}dt,\end{aligned}\end{equation}
where in the last step, we change variable $t\to \frac 1 r$.

The trick to estimate  $\D_{p_n}$ is to calculate
$$g_n(|x|^2):=\int_0^1 e^{-y_n(t)|x|^2}dt.$$
Integrating by part, we have
\[ \begin{aligned}g_n(|x|^2)=&\int_0^1 t' e^{-y_n(t)|x|^2}dt\\=&e^{-|x|^2}+\int_0^1 t y'_n(t)|x|^2e^{-y_n(t)|x|^2}dt\\ =&e^{-|x|^2}
+n(n-1)|x|^2\int_0^1(t^{n-1}-t^n)e^{-y_n(t)|x|^2}dt\\=&e^{-|x|^2}
+n|x|^2\int_0^1(nt^{n-1}-(n-1)t^n)e^{-y_n(t)|x|^2}dt-n|x|^2\int_0^1 t^{n-1}e^{-y_n(t)|x|^2}dt\\=&e^{-|x|^2}
+n|x|^2\int_0^1y_n(t)e^{-y_n(t)|x|^2}dt-|x|^2h_n(|x|^2)\\
=& e^{-|x|^2}
+\frac {\pi n |x|^2}{n-1}\D_{p_n}-|x|^2h_n(|x|^2),\end{aligned}\]
where we denote \begin{equation}\label{hn}h_n(|x|^2):=n\int_0^1 t^{n-1}e^{-y_n(t)|x|^2}dt.\end{equation}
Thus, \begin{equation}\label{firstestimate}\D_{p_n}=\frac{n-1}{n\pi}\left(\frac{g_n(|x|^2)-e^{-|x|^2}}{|x|^2}
+h_n(|x|^2)
\right).
\end{equation}
We claim $$\lim\limits_{n\to\infty}g_n(|x|^2)=1.$$
This is quite straight forward. As $\forall\epsilon\in(0,1)$,
we rewrite,  $$g_n(|x|^2)=\int_0^1 e^{-y_n(t)|x|^2}dt=\int_0^{1-\epsilon}+\int_{1-\epsilon}^1.$$
Since $y_n(t)\to 0$ uniformly on $[0,1-\epsilon]$ as $n\to\infty$, thus
$$\lim_{n\to \infty} \int_0^{1-\epsilon}e^{-y_n(t)|x|^2}dt=\int_0^{1-\epsilon}\lim_{n\to \infty} e^{-y_n(t)|x|^2}dt=1-\epsilon.$$
For the second integration, since $y_n(t)\geq 0$ on $[0,1]$, we have $\int_{1-\epsilon}^1 e^{-y_n(t)|x|^2}dt\leq \epsilon.$
Hence we get $$1-\epsilon\leq\varliminf\limits_{n\to\infty}\int_0^1 e^{-y_n(t)|x|^2}dt\leq\varlimsup\limits_{n\to\infty}\int_0^1 e^{-y_n(t)|x|^2}dt\leq 1.$$
As $\epsilon$ is chosen arbitrarily, letting $\epsilon\to0^+$ yields the claim.

Now we estimate (\ref{firstestimate}) to be \begin{equation}\label{ddd}\begin{aligned}\D_{p_n}=&\frac{n-1}{n\pi}\left(\frac{1-e^{-|x|^2}}{|x|^2}
+h_n(|x|^2)
+o(1)\right)\\=&\frac{1-e^{-|x|^2}}{\pi |x|^2}
+\frac {1} {\pi} h_n(|x|^2)
+o(1)\end{aligned}\end{equation}
as $n \to \infty$.

We now turn to estimate $h_n(|x|^2)$. Change variable $s=t^n$, $h_n$ will be rewritten as
$$\int_0^1e^{z_n(s)|x|^2}ds,$$
where $$z_n(s)=-ns^{\frac{n-1}{n}}+(n-1)s.$$
It's easy to check that $$z_n(s)\leq z_{n+1}(s) $$
 for any fixed $s\in [0,1]$.

Thus we have $z_n(s)$ monotone increasing to $-(s-s\log s)$ as $n\to\infty$, hence, $h_n(|x|^2)$ will satisfy
$$\lim\limits_{n\to \infty}h_n(|x|^2)=\int_0^1e^{-(s-s\log s)|x|^2}ds.
$$
This will give us the estimate $$h_n(|x|^2)=\int_0^1e^{-(s-s\log s)|x|^2}ds+o(1)$$
as $n\to \infty$.

Hence we further estimate (\ref{ddd}) to be
$$\D_{p_n}=\frac{1-e^{-|x|^2}}{\pi|x|^2}+\frac 1 \pi\int_0^1e^{-(s-s\log s)|x|^2}ds
+o(1) \,\,\,\mbox{as}\,\,\, n\to \infty,$$
which completes the proof of Theorem \ref{main}.

\section{Growth of critical values and rescaling limit}\label{for}
By Theorem \ref{main2}, the expected density of the modulus of critical values is $ 1/ x$ decay as $n$ large enough, which implies that the integration of the density over the interval $[a,\infty)$ is infinity for any $a>0$ large enough,  i.e., critical values
accumulate at infinity as $n$ tens to $\infty$. In this section, we will consider the rate of growth of the critical values and its rescaling limit.

 \subsection{Rate of growth}
 We  consider  the distribution function $F_n(x)$ of the following probability density $$\frac{\D_{|p_n|}}{n-1}=\frac1{n-1}\E (\sum_{p_n'=0}\delta_{|p_n|}).$$ We write the distribution function as \begin{equation}\label{distri}F_n(x)=\int_0^x \frac{\D_{|p_n|}}{n-1}dy= \frac{2\pi}{n-1}\int_0^xy\D_{p_n}(y)dy\end{equation}
by relation \eqref{relation}.

Using the identity \eqref{inte} and integrating by part, we will have $$F_n(x)=1-g_n(|x|^2).$$ Now we turn to prove Theorem  \ref{grow}.

\begin{proof}
We only need to prove $g_n(e^{n^{1+\epsilon}})\to 0$ and $g_n(e^{n^{1-\epsilon}})\to 1$. Let's apply the dominated convergence theorem to $g_n(n^{1\pm\epsilon})$, we have
$$\displaystyle \lim\limits_{n \to \infty}g_n(n^{1\pm\epsilon})=\int_0^1 e^{-\lim\limits_{n\to\infty}y_n(t)e^{n^{1\pm\epsilon}}}dt.$$

For '$+$' part, we know that for any fixed $0<t<1$,
$$\lim\limits_{n\to\infty}y_n(t)e^{n^{1+\epsilon}}=\lim\limits_{n\to\infty}(nt^{n-1}-(n-1)t^n)e^{n^{1+\epsilon}}\geq \lim\limits_{n\to\infty}t^{n-1}e^{n^{1+\epsilon}}=\infty.$$

For '$-$' part, we know that $\forall t>0$,
$$\lim\limits_{n\to\infty}y_n(t)e^{n^{1+\epsilon}}\leq \lim\limits_{n\to\infty}nt^{n-1}e^{n^{1-\epsilon}}=0,$$
which implies the conclusions.
\end{proof}


\subsection{Rescaling limit}
As illustrated
by Theorem  \ref{grow}, we need to consider the following rescaled distribution
$$\tilde F_n(x)=F_n(e^{\frac{nx}2})$$
and the corresponding rescaled probability density
$$R_n(x)=(\tilde F_n(x))'=ne^{nx}\int_0^1y_n(t)e^{-y_n(t)e^{nx}}dt$$
by relations \eqref{inte} and \eqref{distri}, where $y_n(t)=nt^{n-1}-(n-1)t^n\in [0,1]$.

Now we prove Theorem \ref{rescaling}.
\begin{proof}

For $x=0$, we have
$$\begin{aligned}\lim_{n\to\infty}R_n(0)= &\lim_{n\to\infty}\int_0^1 ny_n(t)e^{-y_n(t)}dt  \\=& \lim_{n\to \infty} \frac{\pi n}{n-1}\mathbb D_{p_n}(1)= \pi\lim_{n\to \infty}\mathbb D_{p_n}(1),
\end{aligned}$$
where   $\lim_{n\to \infty}\mathbb D_{p_n}(1)$ is a constant given by the leading term in \eqref{mainequation}.

For $x<0$, we have

$$ \begin{aligned}0\leq &\lim_{n\to\infty}\int_0^1 ny_n(t)e^{nx}e^{-y_n(t)e^{nx}}dt\leq \lim_{n\to\infty}\int_0^1 ny_n(t)e^{nx} dt \\ \leq& \lim_{n\to\infty}\int_0^1 n e^{nx} dt \leq \lim_{n\to\infty}ne^{nx} = 0,\end{aligned}$$
which implies \begin{equation}\lim_{n\to\infty}R_n(x)=0\,\,\,\mbox{for}\,\, x<0.\end{equation}

Now we consider the case for $x>0$.  We integrate by part
$$ \begin{aligned}\int_0^1e^{-y_n(t)e^{nx}}dt=&te^{-y_n(t)e^{nx}}|^1_0+\int_0^1y'_n(t)e^{nx}te^{-y_n(t)e^{nx}}dt\\=&e^{-e^{nx}}+
\int_0^1n(n-1)(t^{n-1}-t^n)e^{nx}e^{-y_n(t)e^{nx}}dt\\=&e^{-e^{nx}}+R_n(x)
-\int_0^1nt^{n-1}e^{nx}e^{-y_n(t)e^{nx}}dt\\=&:\, e^{-e^{nx}}+R_n(x)-Q_n(x).\end{aligned}$$
Thus we have
\begin{equation}\label{r}R_n(x)=\int_0^1e^{-y_n(t)e^{nx}}dt+Q_n(x)-e^{-e^{nx}}.\end{equation}
For $x>0$, the third term $e^{-e^{nx}}\to 0$ as $n\to\infty$.

Now we claim:
\begin{equation}\int_0^1e^{-y_n(t)e^{nx}}dt\to e^{-x},\,\,\, Q_n(x)\to 0 \end{equation}
as $n\to\infty$.

If the claim holds, we will get
\begin{equation}\label{s}\lim_{n\to \infty}R_n(x)=e^{-x}\,\,\,\mbox{for}\,\, x>0,\end{equation}
which completes the proof of Theorem \ref{rescaling}.

We now prove the first claim:  For any $t<e^{-x}$, $$\lim\limits_{n\to\infty}y_n(t)e^{nx}\leq \lim\limits_{n\to\infty}nt^{n-1}e^{xn}=0;$$ for any $ t>e^{-x}$, $$\lim\limits_{n\to\infty}y_n(t)e^{nx}\geq \lim\limits_{n\to\infty}t^{n-1}e^{xn}=\infty. $$  Therefore by the dominated convergence theorem, we have $$\displaystyle\lim\limits_{n\to\infty}\int_0^1 e^{-y_n(t)e^{nx}}dt=\int_0^{e^{-x}}1dt=e^{-x}.$$
Now we prove $Q_n(x)\to 0$ as $n\to \infty$. By changing variables, we write $Q_n$ as
 \begin{equation} \begin{aligned}Q_n(x)=&\int_{0}^1n(te^x)^{n-1}e^{-n(te^x)^{n-1}e^x+(n-1)(te^x)^n}d(te^x)\\=&
\int_0^{e^x}nr^{n-1}e^{-nr^{n-1}e^x+(n-1)r^n}dr.
\end{aligned}\end{equation}
We separate the integral $\int_0^{e^x}$ in $Q_n$ to be  $$Q_n(x)=\int_0^{1}+\int_{1}^{e^x}:= I_{1,n}+I_{2,n}.$$
First note $I_{1,n}\geq 0$, $I_{2,n}\geq 0$. We rewrite
\begin{equation} \begin{aligned}I_{1,n}=&
\int_0^{1}nr^{n-1}e^{-nr^{n-1}e^x+(n-1)r^n}dr
\\=&\int_0^1 e^{-nu^{\frac {n-1}n}e^x+(n-1)u}du,\end{aligned}\end{equation}
where we change variable $u=r^n$.

Since $e^x>1$ strictly, we must have $-nu^{\frac {n-1}n}e^x+(n-1)u\to -\infty$ as $n\to\infty$.
By dominated convergent theorem, we have $I_{1,n}\to 0$ as $n\to\infty$.

For the second integration, we further separate
\begin{equation}\label{two}\begin{aligned}I_{2,n}=& \int_1^{1+\epsilon}+\int_{1+\epsilon}^{e^x}=:I_{3,n}+I_{4,n},\end{aligned}\end{equation}
where we choose $1>\epsilon>0$ such that $1+3\epsilon<e^x$.

For the first part, since $e^x>1+3\epsilon>1+\epsilon$, we have
$$r^{n-1}(-ne^x+(n-1)r)\leq r^{n-1}(-nr(1+\epsilon)+(n-1)r)=-r^n(1+n\epsilon),$$
thus
$$I_3\leq \int_1^{1+\epsilon}e^{-r^n(1+n\epsilon)+\log n+(n-1)\log r}dr. $$
But $-r^n+(n-1)\log r\leq 0$ for $r\geq 1$; and $-r^nn\epsilon+\log n\leq -n\epsilon+\log n\to -\infty$ as $n\to\infty$.

Thus $I_{3,n}\leq n\epsilon e^{-n\epsilon}\to 0$ as $n\to\infty$.

For the second part in \eqref{two}, we have
$$\begin{aligned}&\int_{1+\epsilon}^{e^x}nr^{n-1}e^{-nr^{n-1}e^x+(n-1)r^n}dr  \leq &\int_{1+\epsilon}^{e^x}nr^{n-1}e^{- r^{n-1}e^x }dr=\int_{1+\epsilon}^{e^x}e^{(n-1)\log r+\log n-r^{n-1}e^x}dr.
\end{aligned}$$
But for $r\in [1+\epsilon, e^x]$, we have $$-r^{n-1}e^x+(n-1)\log r+\log n\leq -(1+\epsilon)^{n-1}e^x+(n-1)x+\log n\to -\infty$$ as $n\to \infty$.

Hence the $I_{4,n}$ will tend to $0$ as $n\to\infty$. Therefore $I_{2,n}$ tends to $0$ as $n\to\infty$.

Now we must have $\lim_{n\to \infty }Q_n(x)\to 0$, which completes the claim.

\end{proof}

\end{document}